\documentclass[a4paper,reqno,10pt]{amsart}
\usepackage{amssymb,amscd,color}

  \newcommand{\N}{\mathbb{N}}
   \newcommand{\R}{\mathbb{R}}
   
   \newcommand{\C}{\mathbb{C}}
  \newcommand{\Z}{\mathbb{Z}}

  \newcommand{\cU}{\mathcal{U}}

\newtheorem{theorem}{{\bf Theorem}}

\newtheorem{definition}{ {\bf Definition}}

\newtheorem{remark}{Remark}

\newtheorem*{acknowledgements}{Acknowledgements}

\begin{document}
\author{David Mart\'inez Torres}
\address{Departamento de Matemtica, PUC-Rio, R. Mq. S. Vicente 225, Rio de Janeiro 22451-900, Brazil}

\email{dfmtorres@gmail.com}
%
\title{A Poisson manifold of strong compact type}
\begin{abstract}
 We construct a corank one Poisson manifold which is of strong compact type, i.e., the associated Lie algebroid structure on its cotangent bundle is integrable,
 annd the source 1-conected (symplectic) integration is compact. The construction relies on the moduli of marked K3 surfaces. 
\end{abstract}

 \maketitle

\section{Introduction}

Poisson structures make precise the notion of (possibly singular) foliations by symplectic leaves.
A  Poisson structure on a manifold $M$ is given by a bi-vector $\pi$ whose Schouten bracket with itself vanishes,
or more classically, by a bracket  
$\{\cdot, \cdot\}$ on smooth functions such that $(C^\infty(M),\{\cdot,\cdot\})$ becomes  a Lie algebra over the reals 
and $\{f,\cdot\}$ is a derivation for every function $f$. 

 Poisson manifolds may have very complicated characteristic foliation and transverse geometry,
and this results in the lack of global 
or semilocal structure theorems in Poisson geometry.

 A natural way to exert control on a Poisson structure is to require the existence of a well-behaved `Lie group-like'
object  integrating in a suitable sense the infinite dimensional Lie algebra $(C^\infty(M),\{\cdot, \cdot\})$.
To pursue this point of view, it is more convenient to regard at the Lie algebroid structure induced on $T^*M$ by the Poisson structure:
specifically,
the Lie algebroid bracket on 1-forms is determined by the Lie algebra bracket
on functions $[df,dg]:=d\{f,g\}$; the anchor map $\rho$ sends $df$ to its Hamiltonian vector field $X_f:=\{f,\cdot\}$.

Recall that a Lie algebroid $(A\rightarrow M,[\cdot,\cdot],\rho)$ is said to be integrable if there exists a Lie groupoid 
 whose Lie algebroid is isomorphic to $(A\rightarrow M,[\cdot,\cdot],\rho)$; such a Lie groupoid
is called an integration of $(A\rightarrow M,[\cdot,\cdot],\rho)$.  Also, if a Lie algebroid
is integrable, up to isomorphism it has a unique  integration with 1-connected source fibers, its so-called `canonical integration'. 

In analogy with Lie theory and group actions, it is natural to focus one's attention
 on Poisson manifolds whose associated Lie algebroid is integrable,
and the canonical integration is compact.

\begin{definition}\cite{CFM1}
 We say that $(M,\pi)$ is a Poisson manifold of strong compact type (henceforth PMSCT)
 if its associated Lie algebroid structure $(T^*M,[\cdot,\cdot]_\pi,\rho_{\pi})$
 is integrable, and its canonical integration is Hausdorff and compact.
\end{definition}

If the Lie algebroid associated to a Poisson structure is integrable,
then the canonical integration $\Sigma(M)\rightrightarrows M$
carries a canonical
symplectic structure $\Omega_\Sigma$ compatible with the groupoid operations (a multiplicative symplectic structure) \cite{CF,CaF}.
Thus, one can think of a PMSCT as being `desingularized' by a compact symplectic manifold.

By their very definition, PMSCT should have common features with both compact semisimple Lie groups and compact symplectic manifolds, 
and, in fact, it is possible to lay down  a rich general theory for them confirming the above expectation \cite{CFM1}. However, 
 also because the definition of PMSCT is rather demanding,  the construction of PMSCT beyond the trivial
 case of compact symplectic manifolds with finite fundamental group,
 is  a rather challenging problem.  In this respect, in  \cite{CFM1,CFM2} we describe  a connection between 
 PMSCT and quasi-Hamiltonian $\mathbb{T}^n$-spaces.  Quasi-Hamiltonian $\mathbb{T}^n$-spaces appeared for the first time in \cite{MD},
  related to a possible characterization of Hamiltonian circle actions as those symplectic circle actions with fixed points
  (however, the name 
 quasi-Hamiltonian space was coined in \cite{AMM}, where these spaces were analyzed under the perspective of symplectic reduction).
 In \cite{Ko},
  D. Kotschick gave
  an example of a quasi-Hamiltonian $S^1$-space with contractible orbits and free $S^1$-action, answering 
  a question raised in \cite{MS}.
 According to \cite{CFM1}, the Poisson reduced space of such quasi-Hamiltonian $S^1$-space --whose characteristic foliation
 is a fibration over the circle with 
 fiber diffeomorphic to a K3 surface-- must be a PMSCT.

 Kotschick's construction relies on properties  of the moduli space of marked hyperKahler K3 surfaces. 
 In this note we revisit the moduli theory of marked K3 surfaces from a Poisson theoretic perspective, namely, making
 emphasis on the (Poisson) universal family. Then we show how appropriate equivariant mappings into
 the moduli of marked pairs give rise to PMSCT. Finally, we construct such an equivariant mapping 
 building on Kotschick's work (the most delicate part being the existence of `holes' in the moduli of hyperKahler metrics associated
 to roots, an issue which was overlooked in \cite{Ko}).
 
  \begin{acknowledgements}
This note reports on joint work with  M. Crainic and R. L. Fernandes. I would like to thank the referee for his suggestions.
\end{acknowledgements}

\section{Construction of the PMSCT}

The PMSCT is going to be a fibration over the circle with fiber diffeomorphic to the smooth manifold underlying a K3 surface;
the symplectic leaves will be fibers.

\subsection{Families of marked K3 surfaces}

In what follows we will recall the aspects of the moduli theory of marked K3 surfaces, making emphasis
on the universal families. Our reference is  \cite{BPV}, Chapter VIII.

\subsubsection{The objects} A K3 surface is a closed complex surface with trivial canonical bundle and trivial fundamental
group.  K3 surfaces are all diffeomorphic to a fixed 4-dimensional simply connected manifold $F$.
The intersection form in $H^2(F;\Z)$ is isomorphic to the rank 22  lattice \[L:=3H\oplus 2(-E_8),\]
where $H$ is the intersection form of $S^2\times S^2$ and $E_8$ is the intersection form with matrix the absolute values of the
entries of the 
Cartan matrix
of $E_8$ (see \cite{Ki} for background on intersection forms).

A marked K3 surface is defined to be  a K3 surface $S$ together with an isomorphism of lattices $H^2(S;\Z)\rightarrow L$.

\subsubsection{The spaces} 
Let $L_\C$ denote the complexification of the unimodular lattice $L$.

The period domain is defined to be
\[\Omega=\{[v]\in  \mathbb{P}(L_\C)\,|\, \langle v,v \rangle=0,\, \langle v,\bar{v} \rangle >0\}.\]

For marked K3 surfaces there exists a (non-Hausdorff) moduli space $M_1$ together with  a so called  period map
\[\tau\colon M_1\rightarrow \Omega.\] This is an \'etale holomorphic surjection, sending each 
marked K3 in $M_1$ to the line spanned in $L_\C$ by any of its holomorphic symplectic forms (via the marking).

The moduli space relevant for our purposes is the moduli space of marked pairs or marked K3 surfaces together with a Kahler
class/Kahler form, denoted by $K\Omega^0$:

Let $\Delta=\{\delta\in L\,|\,\delta^2=-2\}$ be the so called  set of roots and let 
$L_\R$ be the tensorization of
the unimodular lattice $L$
with the reals. Then
\[K\Omega^0=\{(z,v)\in L_\R\times \Omega\,|\,  \langle z,z \rangle >0,\,z\in v^{\perp},  \langle z,\delta \rangle \neq 0,\,
\forall \delta\in v^{\perp}\cap \Delta\}.\]
There is a forgetful map $K\Omega^0\rightarrow M_1$ which for a fixed complex structure
collapses the Kahler cone to a point; there is also a 
 first projection $\mathrm{pr}_1\colon K\Omega^0\rightarrow L_\R$.

\subsubsection{The families} The two aforementioned moduli spaces have a corresponding universal family.

Firstly, there exist  $\cU\rightarrow M_1$ a universal marked family.

The second universal family is perhaps less standard. We say that $(\mathcal{E}, \pi)\rightarrow W$ is a Poisson structure
of hyperKahler type or a hyperKahler family, if $\mathcal{E}$ is a real analytic family of K3 surfaces and $\pi$ is a (smooth) 
Poisson structure with characteristic foliation  the fibration $\mathcal{E}\rightarrow W$, and 
such that $\pi(w)$ is a Kahler form
for a hyperKahler metric for $E_w$ (here we understand $\pi$ as a fiberwise symplectic form).

The pullback of $\cU$ by the real analytic submersion $K\Omega^0\rightarrow M_1$
gives rise to a real analytic family $K\cU\rightarrow K\Omega^0$. This family is endowed with a
Poisson structure $\pi_{K\cU}$ such that $\pi_{K\cU}(\kappa,v)$ is the Kahler form representing the class $\kappa$ associated to the
unique hyperKahler metric  
whose space of self-dual classes is the positive 3-plane in $L_\R$ spanned by  $\kappa,v$. As a result
$(K\cU, \pi_{K\cU})$ is a universal
marked  Poisson space of hyperKahler type. 

Remark that the integrability criteria in \cite{CF} and \cite{AH} imply that $(K\cU, \pi_{K\cU})$ 
is integrable and that the canonical integration
is Hausdorff.

 \subsubsection{The action} The lattice $L$ has an index 2 subgroup of automorphisms $O^+$: these are the elements
 which act on $L_\R$  preserving the orientation of positive 3-planes \cite{Bo}.
 
 The are obvious actions of $O^+$ on $L_{\R}$ and $\Omega$. The strong form of Torelli's
 theorem says that  $O^+$ also acts on $M_1$  and that the obvious action 
 on $L_\R\times \Omega$ preserves $K\Omega^0$.

The induced action on universal families $\cU$ and $K\cU$ is a consequence on the one hand of the strong form of Torelli's theorem, 
which for each $\gamma\in O^+$ grants the existence of a unique 
biholomorphisms between any two fibers in the orbit of $\gamma$ in the base of the corresponding 
universal family, and the absence of holomorphic vector fields on any K3, which implies that 
the fiberwise biholomorphism  fits into an automorphism (and the latter property also holds for real analytic families \cite{Me}, 
so there is an
induced action on $K\cU$).
It also follows easily that the Poisson tensor $\pi_{K\cU}$ is preserved.

All the maps described between spaces and families are $O^+$-equivariant.

\subsubsection{Universal family of hyperKahler type and PMSCT}

The universal family $(K\cU,\pi_{K\cU})$ provides a way of constructing PMSCT. Before stating our result, recall that an affine subspace $V$
of $H^2(F;\R)$ is called integral affine, if its vector space of directions $\overrightarrow{V}$ intersects the integral
cohomology lattice $H^2(F;\Z)$ in a lattice whose rank is the dimension of $V$ (a lattice of full rank). 

\begin{theorem}\label{thm:const}
Let $V^0$ be a subset of $K\Omega^0$ and $\Gamma$ a subgroup of $O^+$ with the following properties:
\begin{enumerate}
 \item $V^0$ is $\Gamma$-invariant;
 \item  $V:=\mathrm{pr}_1(V^0)$ is an integral affine subspace with a section $f\colon V\rightarrow V^0$;
 \item The action of $\Gamma$ on $V_0$
 is free, proper and co-compact;
\end{enumerate}
Then 
 \[(M,\pi):=f^*(K\cU,\pi_{K\cU})/\Gamma \]
 is a PMSCT.
 \end{theorem}
\begin{proof}
Because $V^0$ is $\Gamma$-invariant the hyperKahler family $f^*(K\cU,\pi_{K\cU})$ is also acted upon by $\Gamma$.
Because the action of $\Gamma$ on $V_0$
 is free, proper and co-compact, the action of $\Gamma$ on $f^*(K\cU,\pi_{K\cU})$ has the same properties. Therefore
 \[(M,\pi):=f^*(K\cU,\pi_{K\cU})/\Gamma\]
 is a well-defined compact Poisson manifold.
 
 The integrability  of a Poisson structure is controlled by the behavior of the so called monodromy lattices $N_{x},x\in M$  \cite{CF}.
  In case the Poisson manifold is a fibration
  with simply connected fiber $F$  and leaf space $B$, near a given fiber $F_{b_0}$ the Poisson structure induces a map
  \[\mathrm{ch}:=U_{b_0}\subset B\longrightarrow H^2(F;\R),\, b\longmapsto [\pi(b)],\]
  and the monodromy lattice 
   at a point $x$ in a fiber $F_b$  is the lattice dual to \[D\mathrm{ch}_b^{-1}(H^2(F;\Z))\subset \nu_x(F_b).\]
   
  Back to our Poisson manifold  $(M,\pi)$, it is identified with $V/\Gamma$. Therefore, for $x\in M$ in the fiber corresponding to $b\in V/\Gamma$,
  the dual of the monodromy lattice $N_x^*$ can be identified with $H^2(F;\Z)\cap \overrightarrow{V}$. Because by assumption 
  $V$ is integral affine, we conclude that $N_x\subset \nu_x(F_b)^*$ is a discrete lattice of full rank.

 It can be proven \cite{CFM2}  that discrete  monodromy lattices of full rank must vary smoothly. Therefore
  $\coprod_{x\in M} N_{x}$ is a closed subset of  the conormal bundle to the fibration, and in particular it is
  uniformly discrete. Hence, according 
  to \cite{CF} $(M,\pi)$ is integrable.
  
  The symplectic fibers of $(M,\pi)$ are simply connected, 
  and therefore the isotropy group of the canonical integration $\Sigma(M)$ at $x$ can be identified with $\nu_x(F_b)^*/N_x$. Because
  $N_x$ has full rank this isotropy group is a torus, and in particular it is compact. As a consequence
  the source fibers of $\Sigma(M)\rightrightarrows M$ are principal bundles with compact fiber over a compact base, and thus
  compact. Since $M$ is compact, we conclude that $\Sigma(M)$ is compact. 
  
 Lastly,
 the delicate issue of the separability of $\Sigma(M)$ follows from the closedness of $\bigsqcup_{x \in M}N_x$ in the conormal bundle
 \cite{AH}.
\end{proof}

 \subsection{Construction of the PMSCT}
 
 We are going to construct a PMSCT with leaf space $\R/\Z$ by applying theorem \ref{thm:const}.  Our construction
 builds on an original idea by Kotschick \cite{Ko}.
 
 The subgroup $\Gamma\cong \Z$ of  $O^+$ is going to be defined by giving a generator $\phi$:
 let $\{u,v\}$,$\{x,y\}$ and $\{z,t\}$ be standard basis of each of the three copies of the hyperbolic intersection form $H$.
Consider the automorphism 
\[\phi\colon 2H\rightarrow 2H,\,\,u\mapsto u,\, v\mapsto v+y,\,x\mapsto x-u,\,y\mapsto y,\]
and extend it to $L$ by the identity on $H\oplus 2(-E_8)$. It is not hard to check that $\phi\in O^+$.

To define the subset $V^0$ we consider the map
\begin{eqnarray*}
f\colon \R &\longrightarrow & L_\R\oplus \mathbb{P}(L_\C)\\
s &\longmapsto & (2u+v+sy,[x-su+2y+e,z+2t+f]),
\end{eqnarray*}
where $e=(e',0),f=(0,e')\in 2(-E_8)$. We assume that the components of $e'$ and 1 are
linearly independent over the rationals, and that $e$ is small enough. We define 
\[V^0:=f(\R).\]

Note that 
\begin{enumerate}
 \item  $V^0$ is $\phi$-invariant and therefore $\Gamma$-invariant;
\item  the projection $\mathrm{pr}_1\colon V^0\rightarrow V$ is a diffeomorphism onto an integral affine line
and therefore it has a section;
\item the action
of $\Gamma$ on $V$  is obviously free, proper and co-compact.
\end{enumerate}
Therefore, to conclude that we are in the  hypotheses of theorem \ref{thm:const}
we just need to make sure that 
\[
 V^0\subset K\Omega^0,
\]
which amounts to solving an arithmetic problem.

The three vectors $2u+v+sy,x-su+2y+e,z+2t+f\in L_\R$ are orthogonal and have positive square provided $e'$ is small. The
difficulty is checking that the positive 3-plane they span is not orthogonal to a root for all $s\in \R$.

We argue by contradiction: we assume that the exist such a root $\delta=\delta_1+\delta_{21}+\delta_{22}$, where
$\delta\in 3H$ and $\delta_{2i}$ belongs to the i-th copy of $-E_8$. 

Because $\delta$ is orthogonal to $2u+v+sy$ we have
\[\delta_1=Ax+Bz-(2D+sA)u+Dv+Ey+Ft,\,A,B,2D+sA,D,E,F\in \Z.\]
Because $\delta$ is orthogonal to $x-su+2y+e,z+2t+f$ we have
\begin{equation}\label{eq:root} F+2B+\langle \delta_{22},e'\rangle-Ds=0,\,E+2A+\langle\delta_{21},e'\rangle=0.
 \end{equation}
The second equation in (\ref{eq:root}) implies $\delta_{21}=0$, because otherwise the irrationality assumption on the
components of $e'$ and 1 would not hold.

Since $-E_8$ is an even, negative definite lattice, we may write $\langle\delta_{22},\delta_{22}\rangle=-2n$, $n\in \N$.
The root condition for $\delta$ becomes 
\begin{equation}\label{eq:root2}2D^2+2A^2+2B^2-\langle B\delta_{22},e'\rangle=1-n.
 \end{equation}
Once more the irrationality assumption on the components of $e'$ and 1 implies that (\ref{eq:root2}) can only have solutions
if $\delta_{22}=0$, and any such solution must also solve 
\[2D^2+2A^2+2B^2=1,\]
which is a contradiction since $A,B,D$ must be integers.

Therefore we conclude that such  root $\delta$ cannot exist and thus $V^0\subset K\Omega^0$,
finishing the construction of the non-trivial PMSCT.

\begin{remark} The previous construction can be arranged to produce PMSCT with 2-dimensional base \cite{CFM3},
  the arithmetic problem associated to the roots being more involved.
\end{remark}

 \end{document}